\documentclass[12pt]{article}
\usepackage{amsfonts,amssymb,amsbsy,amsmath,amsthm,enumerate,verbatim}
\usepackage{color}

\newtheorem{theorem}{Theorem}[section]
\newtheorem{proposition}[theorem]{Proposition}
\newtheorem{lemma}[theorem]{Lemma}

\newtheorem{corollary}[theorem]{Corollary}
\newtheorem{assumption}[theorem]{{\bf Assumption}}

\theoremstyle{definition}
\newtheorem{remark}[theorem]{Remark}

\newcommand{\bel}{\begin{equation} \label}
\newcommand{\ee}{\end{equation}}

\newcommand{\curl}{{\text{curl}}}

\newcommand{\C}{{\mathbb C}}
\newcommand{\R}{{\mathbb R}}
\newcommand{\N}{{\mathbb N}}

\newcommand{\ba}{{\bf a}}

\newcommand{\HH}{{\mathcal H}}

\def\beq{\begin{equation}}
\def\eeq{\end{equation}}
\newcommand{\bea}{\begin{eqnarray}}
\newcommand{\eea}{\end{eqnarray}}
\newcommand{\beas}{\begin{eqnarray*}}
\newcommand{\eeas}{\end{eqnarray*}}
\newcommand{\Pre}[1]{\ensuremath{\mathrm{Re} \left( #1 \right)}}
\newcommand{\Pim}[1]{\ensuremath{\mathrm{Im} \left( #1 \right)}}
{

\begin{document}

\begin{center}
{\Large \bf Stability estimate in an inverse problem for non autonomous magnetic Schr\"odinger equations}
\end{center}



\medskip

\begin{center}
\footnote{Laboratoire d'Analyse, Topologie, Probabilit\'es, 39, rue F. Joliot Curie, 13453 Marseille, France.}{\bf Michel Cristofol}, \footnote{Centre de Physique Th\'eorique, CNRS-Luminy, Case 907, 13288 Marseille, France.}{\bf Eric Soccorsi}
\end{center}

\medskip
\begin{abstract}
We consider the inverse problem of determining the time dependent magnetic field of the Schr\"odinger equation
in a bounded open subset of $\R^n$, $n \geq 1$, from a finite number of Neumann data, when the boundary measurement is taken on an appropriate open subset of the boundary. We prove the Lispchitz stability of the magnetic potential in the Coulomb gauge class by $n$ times changing initial value suitably.
\end{abstract}



 \bigskip

\section{Introduction}
\subsection{Statement of the problem}
\label{sec-wwaaf}
\setcounter{equation}{0}
Let $T>0$, $n \in {\mathbb N}^*$, and let $\Omega$ be a bounded open subset of ${\mathbb R}^n$ with ${\rm C}^2$-boundary $\Gamma$. We consider the time-dependent Hamiltonian
$H_{\ba}(t) := ({\rm i} \nabla + \chi(t) \ba )^2$
associated to the nondivergent magnetic vector potential $\chi(t) \ba(x)$, where $\chi$ is a smooth real valued function on $[0,T]$ and $\ba \in {\rm H}^1(\Omega)^n$ is bounded, together with the related Schr\"odinger equation,
\bel{ia1}
\left\{  \begin{array}{ll} -{\rm i} u'(t,x) +  H_{\ba}(t) u(t,x) = 0, & (t,x) \in Q_T^+ := (0,T) \times \Omega  \\ u(t,x) = 0, & (t,x) \in \Sigma_T^+ := (0,T) \times \Gamma  \\ u(0,x) = u_0(x), & x \in \Omega, \end{array} \right.
\ee
for some suitable data $u_0$. Here and throughout all this text, $u'(t,x)$ stands for $\partial_t u(t,x)$ and $\nabla u(t,x) := (\partial_{x_1} u(t,x),\ldots,\partial_{x_n} u(t,x))$ is the gradient of $u(t,.)$ at $x=(x_1,x_2,\ldots,x_n) \in \Omega$. As follows from Remark \ref{rem-rvmp} below, we may assume in the sequel with no loss of generality that the magnetic potential vector $\ba$ is real-valued.

Let $\nu$ denote the unit outward normal vector to $\Gamma$ and set $\partial_{\nu} v := \nabla v . \nu$. Then, $\Gamma^+$ denoting an open subset of $\Gamma$ satisfying an appropriate geometrical condition we shall make precise further, we aim to retrieve $\ba=(\ba_j(x))_{1 \leq j \leq n}$, $x\in \Omega$, in \eqref{ia1}, from the extra data $\left. \partial_{\nu} \partial_t^k u \right|_{(0,T) \times \Gamma^+}$,
$k=1,2$, by $n$ times changing initial value $u_0$ suitably. Hence we investigate the problem to know whether a finite number of partial Neumann data of \eqref{ia1} (in absence of any Cauchy lateral data information given by the Dirichlet to Neumann map, denoted by ``DN map" for short in the sequel, $\Lambda_{\ba}$, associated to $\ba$) determines uniquely the magnetic potential $\ba$ in the Coulomb gauge class (i.e. the class of divergence free vectors in $\Omega$).


\subsection{Existing papers}
Numerous papers establishing the uniqueness of inverse problems coefficients from the DN map (or scattering information) have actually been published over the last years.
In the particular case of the magnetic Schr\"odinger equation, it is noted in \cite{Eskin0} that the DN map is invariant under the gauge transformation of the magnetic potential, i.e. $\Lambda_{\ba + \nabla \psi}= \Lambda_{\ba}$ when $\psi \in {\rm C}^1(\overline{\Omega})$ is such that $\psi_{| \Gamma}=0$.
Therefore the magnetic potential cannot be uniquely determined from the DN map (we can at best expect uniqueness modulo a gauge transform of $\ba$ from $\Lambda_{\ba}$).
However the magnetic field ${\rm d} \ba$, where ${\rm d} \ba$ is the exterior derivative of $\ba$ interpreted as the 1-form $\sum_{j=1}^n \ba_j {\rm d} x_j$, is preserved. If $n=3$ then ${\rm d} \ba$ corresponds to ${\rm \curl}\ \ba$.
Conversely it is shown in \cite{Eskin1} for multiconnected domains, that if the DN maps $\Lambda_{\ba}$ and $\Lambda_{\tilde{\ba}}$ are gauge equivalent (i.e. ${\rm e}^{-{\rm i} \psi} \Lambda_{\ba} {\rm e}^{{\rm i} \psi} = \Lambda_{\tilde{\ba}}$) then $\ba$ and $\tilde{\ba}$ are gauge equivalent too.

Actually Z. Sun proved in \cite{Sun} that the DN map determines the magnetic field provided $\ba$ is small in an appropriate class. In \cite{NSU} the smallness assumption was removed for ${\rm C}^{\infty}$ magnetic potentials. The regularity assumption on $\ba$ was weakened to ${\rm C}^1$ in \cite{Tol} and to Dini continuous in \cite{Sal}, and the
uniqueness result was extended to some less regular but small potentials in \cite{Panch}. Recently in \cite{BeCh10}, M. Bellassoued and M. Choulli proved that the magnetic field depends stably on the dynamical DN map.
In \cite{Eskin2}, G. Eskin considered the inverse boundary value problem for the Schr\"odinger equations with electromagnetic potentials, in domains with several obstacles. He proved the uniqueness modulo a gauge transform of the recovery of the potentials from the DN map, under geometrical conditions on the obstacles. In \cite{Sal1}, M. Salo reconstructed the magnetic field from the DN map using semiclassical pseudodifferential calculus and the construction of complex geometrical optics solutions. All the above cited papers considered time independent magnetic potentials. The uniqueness in the determination from the DN map, of time dependent magnetic potentials appearing in a Schr\"odinger equation (in a domain with obstacles), was proved by G. Eskin in \cite{Eskin3}. The main ingredient in his proof is the construction of geometric optics solutions. As far as we know, this is the only existing paper dealing with the determination of a time dependent magnetic potential in a Schr\"odinger equation.

All the above mentioned results were obtained with the full data, i.e. measurements of the DN map are made on the whole boundary. The uniqueness problem by a local DN map was solved by D. Dos Santos Ferreira, C. E. Kenig, J. Sj\"ostrand and G. Uhlmann in \cite{FKSU}. Similarly it was shown in \cite{Tzou} that the magnetic field depends stably on the DN map measured on any subboundary $\Gamma_0$ which is slightly larger than half the boundary. This result was extended in \cite{BenJ} to arbitrary small $\Gamma_0$ provided the magnetic potential is known near the boundary.

Notice that infinitely many observations were required in all the above cited results.
To our knowledge (and despite of the fact that this seems more suited to numerical applications) there is no result available in the mathematical literature on the recovery of a magnetic potential appearing in a Schr\"odinger equation, from a finite number of boundary measurements. Nevertheless this is not the case for external electric potentials. Indeed, the problem of stability in determining the time independent electric potential in a Schr\"odinger equation from a single boundary measurement was treated by L. Baudouin and J.-P. Puel in \cite{BP}. This result was improved by A. Mercado, A. Osses and L. Rosier in \cite{MAR}. In the two above mentioned papers, the main assumption is that the part of the boundary where the measurement is made satisfies a geometric condition related to geometric optics condition insuring observability. This geometric condition was relaxed in \cite{BC} under the assumption that the potential is known near the boundary.

In the present article we prove Lipschitz stability in the recovery of the time dependent magnetic potential appearing in the Schr\"odinger equation, from a finite number of observations $\partial_{\nu} \partial_t^k u$, $k=1,2$, measured on a subboundary for different choices of $u_0$ in \eqref{ia1}, by a method
based essentially on an appropriate Carleman estimate. We refer to \cite{Albano}, \cite{BP} and \cite{Tataru} for actual examples of this type of inequalities for the Schr\"odinger equation. The original idea of using a Carleman estimate to solve inverse problems goes back to the pioneering paper \cite{BK:81} by A. L. Bugkheim and M. V. Klibanov. This technique has then been widely and successfully used by numerous authors (see e.g. \cite{ImaYama98}, \cite{ChoulliYama99}, \cite{IsaYama00}, \cite{BukhChengIsaYama01}, \cite{BP}, \cite{Ima02}, \cite{Bel04},  \cite{ChoulliYama04}, \cite{Bellayama06}, \cite{CGR}, \cite{KliYama06}, \cite{Bellayama07},   \cite{CR}, \cite{BCGY09}
and references therein) in the study of inverse wave propagation, elasticity, or parabolic problems. However, due to the presence of time dependent coefficients (involving the magnetic potential vector $\ba$ we aim to retrieve) for zero  and first order space derivatives  in the expression of $H_{\ba}(t)$, the solution to the inverse problem we address in this text cannot be directly adapted from the above references.

\subsection{Main results}
\label{intro-mr}
In this section we state the main result of this article and briefly comment on it.

Choose $\ba_0 \in {\rm H}^1(\Omega)^n \cap {\rm H}_{{\rm div}0}(\Omega;\R)$, where ${\rm H}_{{\rm div}0}(\Omega;\R):= \{ \ba \in {\rm L}^{\infty}(\Omega;\R^n),\ \nabla . \ba = 0 \}$ is the space of real valued and bounded magnetic potentials vectors in the Coulomb gauge class, and define the set of ``admissible potential vectors" as
$${\bf A}(\ba_0,M) := \{ \ba \in {\rm H}^1(\Omega)^n \cap {\rm H}_{{\rm div}0}(\Omega;\R),\  \| \ba \|_{{\rm L}^{\infty}(\Omega)^n} \leq  M\ {\rm and}\ \ba(\sigma) = \ba_0(\sigma)\ {\rm a.e.}\ \sigma \in \Gamma \}.$$
By selecting $\ba$ in ${\bf A}(\ba_0,M)$ we enjoin fixed value to $\ba$ on the boundary, which is nothing else but the measurement on $\Gamma$ of the magnetic potential  we want to determine. Similar (or even stronger) ``compatibility conditions" imposed on inverse problems coefficients have already been used in various contexts in e.g. \cite{AU}, \cite{CR}, \cite{Eskin3} or \cite{BC}.

The main result of this paper is the following global stability estimate for magnetic potential vectors in ${\bf A}(\ba_0,M)$.

\begin{theorem}
\label{thm-stab}
Let $T>0$, $n$, $\Omega$ and $\Gamma$ be the same as in \S \ref{sec-wwaaf}.
Let $\Gamma^+ \subset \Gamma$ fulfill the geometrical condition of Assumption \ref{funct-beta}, let $\ba_0 \in {\rm H}^1(\Omega)^n \cap {\rm H}_{{\rm div}0}(\Omega;\R)$, and let $\chi \in {\rm C}^3([0,T];\R)$ be such that $\chi(0)=0$ and $\chi'(0) \neq 0$. Pick $n$ functions $u_{0,j} \in {\rm H}_0^{\max(6,n \slash 2 + 1 + \varepsilon)}(\Omega;\R)$, $j=1,\ldots,n$, for some $\varepsilon>0$, satisfying:
\bel{icc}
{\rm det}\ DU_0(x)  \neq 0,\ x=(x_1,\ldots,x_n) \in \Omega,\ {\rm where}\ DU_0(x):=(\partial_{x_j} u_{0,i}(x) )_{1 \leq i, j \leq n}.
\ee
Let $\ba$ (resp. $\tilde{\ba}$) be in ${\bf A}(\ba_0,M)$, and let $u_j$ (resp. $\tilde{u}_j$) denote the ${\rm C}^0([0,T];{\rm H}_0^1(\Omega) \cap {\rm H}^2(\Omega))$-solution to \eqref{ia1} (resp. \eqref{ia1} where $H_{\ba}(t)$ is replaced by $H_{\tilde{\ba}}(t)$) with initial condition $u_{0,j}$, $j=1,\ldots,n$. Then there exists a constant $C>0$, depending only on $T,\Omega,\Gamma^+,M,\chi$, and $\{ u_{0,j} \}_{j=1}^n$, such that we have:
$$ \| \tilde{\ba}-\ba \|_{{\rm L}^2(\Omega)^n}
\leq C \sum_{j=1}^n \left( \| \partial_{\nu} \partial_{t} (u_j - \tilde{u}_j) \|_{{\rm L}^2(0,T;\Gamma^+)}^2 +\| \partial_{\nu} \partial_{t}^2 (u_j - \tilde{u}_j) \|_{{\rm L}^2(0,T;\Gamma^+)}^2 \right). $$
\end{theorem}

This immediately entails the:

\begin{corollary}
Under the conditions of Theorem \ref{thm-stab}, the following implication holds true for every $\ba$ and $\tilde{\ba}$ in ${\bf A}(\ba_0,M)$:
$$ \left( u_j(t,x) = \tilde{u}_j(t,x),\  j=1,\ldots,n,\ {\rm a.e.}\ (t,x) \in (0,T) \times \Gamma^+ \right)
\Rightarrow \left( \ba = \tilde{\ba} \right). $$
\end{corollary}

These results suggest several comments.
\begin{enumerate}
\item  We emphasize that the stability result of Theorem \ref{thm-stab} only requires a finite number (directly proportional to the dimension of $\Omega$) of observations on an appropriate subboundary, for a finite span-time. This result actually involves less measurements than the other reconstruction methods for magnetic potential vectors (or magnetic fields) known so far.
\item Although the inverse problem examined in \cite{BP}, of determining the external potential $q$ appearing in the Schr\"odinger equation $-{\rm i} u' + \Delta u + q u =0$ in $Q_T$, where $\Delta u := \sum_{j=1}^n \partial_{x_j}^2 u$ is the Laplacian of $u$, from the single measurement $\left. u \right|_{\Sigma_T^+}$, might seem very similar to the one addressed in this paper, there is an additional major mathematical difficulty when dealing with the magnetic potential. This comes from the presence of first order spatial differential terms of the form $\ba . \nabla u$ in the effective Hamiltonian $H_{\ba}(t)$.
\item If $n=3$, the equations in \eqref{ia1} model (in a ``natural" system of units where the various physical constants are taken equal to 1) the evolution of the wave function $u$ of a charged particle subject to the action of the magnetic field ${\bf b}(t,x) := \chi(t)\ {\rm curl}\ \ba(x)$, starting from the state $u_0$. Notice that in classical physics, only ${\bf b}$ has a physical meaning, and $\ba$ is a mathematical tool only. Moreover, in light of \cite{HHHO}[Remark 2.1], one can always assume (at least for a sufficiently smooth $\Gamma$) that $\ba$ satisfies the Coulomb gauge condition $\nabla . \ba=0$. As already mentioned in \S \ref{sec-wwaaf}, there is an obstruction to the uniqueness of $\ba$, since the magnetic potential appearing in the Schr\"odinger equation can at best be recovered modulo a gauge transform. The Coulomb gauge condition imposed on $\ba$ actually eliminates this indetermination.
\item In the particular case where $\chi(t)=\sin ( \omega t )$ for some $\omega>0$, \eqref{ia1} describes the evolution of the wave function $u$ of a (periodically) laser pulsed charged particle (see e.g. \cite{DL5}[Chap XVIII,\S 1.4.2.2]).
\item The Lipschitz stability inequality stated in Theorem \ref{thm-stab} is not only interesting from the mathematical point of view, but may also be a very useful tool in view of numerical simulations (see \cite{Liyamazou09} where this point is discussed).
\end{enumerate}

\subsection{Contents}
The paper is organized as follows. In section \ref{sec-tdmh} we study the solutions to non autonomous magnetic Schr\"odinger equations. Existence, uniqueness and regularity results are stated in \S \ref{ssec-dual}, and we establish in \S \ref{sec-cce} that the charge and the energy of these systems remain uniformly bounded in the course of time. The main properties of the associated magnetic Hamiltonians, needed in the proofs, are collected in \S \ref{sec-mso} and \S \ref{sec-tdh}. Section \ref{sec-si} is devoted to establishing the stability inequality stated in Theorem \ref{thm-stab}.
The inverse problem of determining the magnetic potential vector from partial lateral Neumann data is discussed in \S \ref{sec-ip}. The strategy used essentially relies on a global Carleman estimate for magnetic Schr\"odinger equations, given in \S \ref{sec-ce}. The proof of the global stability inequality is given in \S \ref{sec-li}.

\section{Non autonomous magnetic Schr\"odinger equations}
\label{sec-tdmh}
\setcounter{equation}{0}
The system under study is modeled by non autonomous magnetic Schr\"odinger equations which do not fit the usual requirements (see e.g. \cite{LM}[Chap. 5, Assumption (12.1)] imposing time independent coefficients to (non zero) order space derivatives) of the classical existence, uniqueness or regularity results in ${\rm L}^2(Q_T^+)$, found in the mathematical P.D.E. literature.
Furthermore, the strategy used in Section \ref{sec-si} requires that the system \eqref{ia1} be differentiated twice w.r.t. $t$, and, roughly speaking, then solved in a wider space of the form ${\rm L}^2(0,T;\HH')$, where $\HH'$ denotes the dual space of ${\rm H}_0^1(\Omega)$ or ${\rm H}_0^1(\Omega) \cap {\rm H}^2(\Omega)$. This statement is made precise in Lemma \ref{lm-r2}.
The corresponding existence, uniqueness and regularity results on the solutions to these problems are collected in Lemma \ref{lm-r1}, while the question of their equivalence in the above various functional spaces is treated by Lemma \ref{lm-unique}. Finally, some {\it a priori} estimates used in the derivation of the stability equality of Theorem \ref{thm-stab}, are derived in Lemma \ref{lm-ec} and Proposition \ref{pr-ec}.

\subsection{Notations}
In this subsection we introduce some basic notations used throughout the article.
Let $X_1$, $X_2$ be two separable Hilbert spaces. We denote by $\mathcal{B}(X_1,X_2)$ the class of linear bounded operators $T:X_1 \rightarrow X_2$.
Let $(a,b)$ be an open set of $\mathbb{R}$.
If the injection $X_1 \hookrightarrow X_2$ is continuous and $X_1$ is dense in $X_2$, we define
$W(a,b;X_1,X_2) := \{u;\ u \in {\rm L}^2(a,b;X_1),\ u' \in {\rm L}^2(a,b;X_2) \}$,
which, endowed with the norm $\| u \|_W:=(\| u \|_{{\rm L}^2(a,b;X_1)}^2 + \| u' \|_{{\rm L}^2(a,b;X_2)}^2)^{1 \slash 2}$, is a Hilbert space. More generally we put $W^m(a,b;X_1,X_2) := \{u;\ u \in {\rm H}^m(a,b;X_1),\ u'  \in {\rm H}^m(a,b;X_2) \}$ for every $m \in \N^*$, where ${\rm H}^m(a,b;X_1)$ denotes the usual $m^{{\rm th}}$-order Sobolev space of $X_1$-valued functions. For the sake of convenience we set ${\rm H}^0(a,b;X_1) := {\rm L}^2(a,b;X_1)$ in such a way that $W^0(a,b;X_1,X_2):=W(a,b;X_1,X_2)$. If $X_1=X_2=X$ we write $\mathcal{B}(X)$ instead of $\mathcal{B}(X,X)$, and $W^m(a,b;X)$ instead of $W^m(a,b;X,X)$, $m \in \N$.\\


\subsection{Magnetic Schr\"odinger operators}
\label{sec-mso}
Let $\ba \in {\rm C}^0([0,T];{\rm H}_{{\rm div}0}(\Omega;\R))$.
We consider the linear self-adjoint operator $H(t)$, $t \in [0,T]$ in $\HH_0:={\rm L}^2(\Omega)$, associated with the closed, densely defined and positive sesquilinear form
\bel{r1}
h(t;u,v) := \langle ({\rm i} \nabla + \ba(t,x)) u, ({\rm i} \nabla + \ba(t,x)) v \rangle_0,
\ee
with a domain independent of $t$,
${\rm dom}\ h(t)={\rm dom}\ h(0) = {\rm H}_0^1(\Omega)$. Here $\langle . , . \rangle_0$ denotes the standard scalar product in $\HH_0^n$.
The space ${\rm dom}\ h(0)$ endowed with the scalar product $\langle u , v \rangle_1 := \langle (-\Delta+1)^{1 \slash 2} u, (-\Delta+1)^{1 \slash 2} v \rangle_0$ (or equivalently  $\langle u , v \rangle_1= \langle u , v \rangle_0 + \langle \nabla u , \nabla v \rangle_0$) is a Hilbert space denoted by $\HH_1$, and we have
\bel{r1a}
\langle H(t) u , v \rangle_0 = h(t;u,v),\ u \in {\rm dom}\ H(t),\ v \in \HH_1.
\ee
The boundary $\Gamma$ being ${\rm C}^2$, we actually know from \cite{Caz}[Chap. 2] that
\bel{r1b}
{\rm dom}\ H(t)={\rm dom}\ (-\Delta) = {\rm H}_0^1(\Omega) \cap {\rm H}^2(\Omega).
\ee
We call $\HH_{-1}$ the dual space of $\HH_1$, that is to say the
vector space of continuous conjugate linear forms on $\HH_1$. For any
$u \in \HH_0$, the functional $v \mapsto \langle u , v \rangle_0$ belongs
to $\HH_{-1}$ since
$|\langle u , v \rangle_0 | \leq \|u\|_0 \|v\|_0
\leq \|u\|_0 \| v \|_1$,
and we can also regard $\HH_0$ as a subspace of $\HH_{-1}$.
Hence $\HH_1 \subset \HH_0 \subset \HH_{-1}$
where the symbol $\subset$ means a topological embedding, and
$H(t)$, $t \in [0,T]$, can be extended into an operator mapping $\HH_1$ into
$\HH_{-1}$ since
$$
  | h(t;u,v) | \leq C_T^2 \|u\|_{1} \| v \|_1,\ u, v \in \HH_1,
$$
where
\bel{r1b2}
C_T := (1 + n \mathcal{A}_0^2)^{1 \slash 2},\ \mathcal{A}_0 := \| \ba \|_{{\rm C}^0([0,T];{\rm L}^{\infty}(\Omega)^n)} = \sup_{t \in [0,T]} \| \ba(t,.) \|_{{\rm L}^{\infty}(\Omega)^n}.
\ee
Let us denote by $\langle \cdot , \cdot \rangle_{-1,1}$ the dual pairing
between $\HH_{-1}$ and $\HH_{1}$. This pairing is linear in
the first and conjugate linear in the second argument. In other words, the
embedding $\HH_0 \subset \HH_{-1}$ means that
$\langle g , \psi  \rangle_{-1,1}=\langle g , \psi  \rangle_0$ for all
$g \in \HH_0$ and $\psi \in \HH_1$, and the mapping
$H(t):\HH_{1} \to \HH_{-1}$ is defined so that
$\langle H(t) u , v \rangle_{-1,1}=h(t;u,v)$ for all $u,v \in \HH_1$.

The space ${\rm dom}\ (-\Delta)$, endowed with the scalar product $\langle u , v \rangle_2 := \langle (-\Delta+1) u, (-\Delta+1) v \rangle_0$, is an Hilbert space denoted by $\HH_2$.
For all $u \in \HH_2$, we have
\bel{r1d}
H(t) u = ({\rm i} \nabla + \ba(t))^2 u = (-\Delta + 2 {\rm i} \Pre{\ba(t)} . \nabla + |\ba(t)|^2) u,
\ee
from \eqref{r1a}-\eqref{r1b} and the Coulomb gauge condition $\nabla . \ba = 0$, which entails
\bel{r1c}
\| H(t) u \|_0 \leq \| \Delta u \|_0 + 2 \mathcal{A}_0 \| \nabla u \|_0 + n \mathcal{A}_0^2 \| u \|_0 \leq 2 C_T^2 \| u \|_2.
\ee
We call $\HH_{-2}$ the dual space of $\HH_2$. Similarly, we have
$\HH_2 \subset \HH_1 \subset \HH_0 \subset \HH_{-1} \subset \HH_{-2}$,
and we deduce from \eqref{r1c} that $H(t)$ can be extended into an operator mapping $\HH_0$ into $\HH_{-2}$. Denoting
$\langle . , . \rangle_{-2,2}$ the dual pairing between $\HH_{-2}$ and $\HH_2$, the mapping $H(t) : \HH_0 \to \HH_{-2}$ is defined by
$$\langle H(t) u , v \rangle_{-2,2} := \langle u , H(t) v \rangle_0,\ u \in \HH_0,\ v \in \HH_2.$$

\begin{remark}
\label{rem-rvmp}
In light of \eqref{r1d} we may assume in the sequel without limiting the generality of the foregoing that the magnetic vector potential $\ba$ is real-valued.
\end{remark}


\subsection{Time-dependent magnetic Hamiltonians}
\label{sec-tdh}
In this section we examine the dependence of $H(t)$ w.r.t. $t$ and establish some of the properties of its derivatives which will play a crucial role in this paper.

For $j=1,2,3$ we assume that $\ba \in {\rm C}^j([0,T];{\rm H}_{{\rm div}0}(\Omega;\R))$ and
introduce the operator
$$B_j(t):= \ba^{(j)}(t) . ( {\rm i} \nabla + \ba(t) ),\ \ba^{(j)}(t):=\frac{{\rm d}^j \ba(t)}{{\rm d}t^j},$$
with domain $\HH_1$, acting and densely defined in $\HH_0$.
In light of the divergence free condition imposed on $\ba$, we have
\bel{r12}
\langle B_j(t) p, q \rangle_0 = \langle p, B_j(t) q \rangle_0,\ p,q \in \HH_1,
\ee
thus $B_j(t)$ is symmetric in $\HH_0$. Further, for all $u \in \HH_1$, we have
\bel{r12star}
\| B_j(t) u \|_0 \leq \mathcal{A}_j \| ( {\rm i} \nabla + \ba(t)) u \|_0 \leq \mathcal{A}_j C_T \| u \|_1,
\ee
where $C_T$ is the same as in \eqref{r1b2}, and
\bel{r12c}
\mathcal{A}_j:= \| \ba \|_{{\rm C}^j([0,T];{\rm L}^{\infty}(\Omega)^n)} =\sup_{t \in [0,T]} \| \ba^{(j)}(t) \|_{{\rm L}^{\infty}(\Omega)^n},\ j=1,2,3.
\ee
Hence $B_j(t)$, $j=1,2,3$, can be extended from $\HH_0$ into $\HH_{-1}$, by setting:
\bel{r12b}
\langle B_j(t) u , v \rangle_{-1,1} := \langle u , B_j(t) v \rangle_0,\ u \in \HH_0,\ v \in \HH_1.
\ee
For all $t \in [0,T]$, we then define the three following operators
\bel{r12a}
H^{(1)}(t):= 2 B_1(t),\
H^{(2)}(t):= 2 [B_2(t) +  \ba'(t)^2]\ {\rm and}\ H^{(3)}(t):= 2 [B_3(t) + 3 \ba'(t) . \ba''(t) ]
\ee
with domain $\HH_1$.
In light of \eqref{r12} the operators $H^{(j)}(t)$, $j=1,2,3$, are symmetric in $\HH_0$ for each $t \in [0,T]$, and, due to \eqref{r12star}, there exists moreover a constant $\ell_j>0$, depending only on $C_T$ and $\{ \mathcal{A}_k \}_{k=0}^j$, such that we have:
\bel{r14}
\| H^{(j)}(t) u \|_0 \leq \ell_j \| u \|_1,\ u \in \HH_1.
\ee
Hence $H^{(j)}(t)$ can be extended from $\HH_0$ into $\HH_{-1}$ as
$\langle H^{(j)}(t) u , v \rangle_{-1,1} = \langle u , H^{(j)}(t) v \rangle_0$ whenever $u \in \HH_0$ and $v \in \HH_1$. Moreover for $\ba$ in ${\rm C}^j([0,T];{\rm H}_{{\rm div}0}(\Omega;\R))$, $j=1,2,3$, the mapping $t \mapsto h(t;u,v)$, for $u,v \in \HH_1$ fixed, is $j$ times continuously differentiable on $[0,T]$, and it holds true that
\bel{r13}
\frac{{\rm d}^j}{{\rm d}t^j} h(t;u,v)  = h^{(j)}(t;u,v)  = \langle H^{(j)}(t) u, v \rangle_0 = \langle u, H^{(j)}(t) v \rangle_0,\ j=1,2,3.
\ee
When convenient, we shall write in the sequel $H'(t)$ for $H^{(1)}(t)$, $H''(t)$ instead of $H^{(2)}(t)$, and $H'''(t)$ for $H^{(3)}(t)$.

\subsection{Solution to non-autonomous Schr\"odinger equations}
\label{ssec-dual}
Throughout this section we assume that $\ba \in {\rm C}^1([0,T];{\rm H}_{{\rm div}0}(\Omega;\R))$.

\subsubsection{Solutions in ${\rm L}^2(0,T;\HH_{-j})$, $j=0,1,2$.}
Let $f$ be in ${\rm L}^2(0,T;\HH_{-j})$ for $j=0,1,2$. A solution to the Schr\"odinger equation
\bel{a1}
- {\rm i} \psi' + H(t) \psi = f\ {\rm in}\ {\rm L}^2(0,T;\HH_{-j}),
\ee
is a function $\psi \in {\rm L}^2(0,T;\HH_{2-j})$ satisfying for every $v \in \HH_j$:
\bel{a2}
- {\rm i} \frac{{\rm d}}{{\rm d} t} \langle \psi(t) , v \rangle_0 + \langle H(t) \psi(t) , v \rangle_{-j,j} = \langle f(t), v \rangle_{-j,j}\ {\rm in}\ {\rm C}_0^{\infty}(0,T)',
\ee
where $\langle . , . \rangle_{-j,j}$ stands for $\langle . , . \rangle_0$ in the particular case where $j=0$.

Let us now introduce the following notations that will be used in the remaining of this section.
For all $v \in C_0^{\infty}(\Omega)=C_0^{\infty}(\Omega;{\mathbb C})$ and $\varphi \in C_0^{\infty}(0,T)=C_0^{\infty}(0,T;{\mathbb C})$ we set
$$(v \otimes \varphi)(t,x) := v(x) \varphi(t),\ t \in (0,T),\ x \in \Omega, $$
and define the space ${\rm C}_0^{\infty}(\Omega) \otimes {\rm C}_0^{\infty}(0,T):=\{ v \otimes \varphi,\ v \in {\rm C}_0^{\infty}(\Omega),\ \varphi \in {\rm C}_0^{\infty}(0,T) \}$.
We recall that
\bel{w0}
{\rm C}_0^{\infty}(\Omega) \otimes {\rm C}_0^{\infty}(0,T)\ {\rm is\ dense\ in}\ {\rm C}_0^{\infty}(Q_T^+)={\rm C}_0^{\infty}(Q_T^+;{\mathbb C}),\ Q_T^+ :=(0,T) \times \Omega.
\ee
Let $\psi$ be solution to \eqref{a1}. For all $v \in C_0^{\infty}(\Omega)$ and $\varphi \in C_0^{\infty}(0,T;\R)$, it follows from \eqref{a2} that
\bel{a2b}
{\rm i}  \int_0^T \langle \psi(t), v \rangle_0 \varphi'(t) dt = \int_0^T \langle g(t), v \rangle_{-j,j} \varphi(t) dt, \ee
where $g(t):=f(t) -H(t) \psi(t)$ for a.e. $t \in (0,T)$. Evidently $g \in {\rm L}^2(0,T;\HH_{-j})$, and
\eqref{a2b} yields
\bel{a3}
{\rm i}  \int_0^T \langle \psi(t), (v \otimes \varphi)'(t) \rangle_0 dt = \int_0^T \langle g(t), (v  \otimes \varphi)(t) \rangle_{-j,j} dt,
\ee
for all $v \otimes \varphi \in {\rm C}_0^{\infty}(\Omega) \otimes {\rm C}_0^{\infty}(0,T)$.
Since $\psi \in {\rm L}^2(0,T;\HH_0)$, $g \in {\rm L}^2(0,T;\HH_{-j})$, and ${\rm C}_0^{\infty}(\Omega) \otimes {\rm C}_0^{\infty}(0,T) \subset {\rm C}_0^{\infty}(Q_T^+)$, \eqref{a3} can be rewritten as
$$
{\rm i}  \langle \psi, (v \otimes \varphi)' \rangle_{{\rm C}_0^{\infty}(Q_T^+)',{\rm C}_0^{\infty}(Q_T^+)}  = \langle g(t), v \otimes \varphi \rangle_{{\rm C}_0^{\infty}(Q_T^+)',{\rm C}_0^{\infty}(Q_T^+)},\ v \otimes \varphi \in {\rm C}_0^{\infty}(\Omega) \otimes {\rm C}_0^{\infty}(0,T).
$$
From this and \eqref{w0} then follows that $i \psi' = g$ in ${\rm C}_0^{\infty}(Q_T^+)'$. Taking into account that $g \in {\rm L}^2(0,T;\HH_{-j})$ we obtain that
$\psi' \in {\rm L}^2(0,T;\HH_{-j})$
and thus that \eqref{a1} holds true.
Henceforth any solution to \eqref{a1} belongs to
$W(0,T;\HH_{2-j},\HH_{-j})$ and
\eqref{a2} can be rewritten as
\bel{a5}
- {\rm i} \langle \psi'(t) , v \rangle_{-j,j} + \langle H(t) \psi(t) , v \rangle_{-j,j} = \langle f(t), v \rangle_{-j,j}\ {\rm in}\ {\rm C}_0^{\infty}(0,T)',\ v \in \HH_j.
\ee
As a consequence a solution to the Schr\"odinger equation \eqref{a1} is a function $\psi \in W(0,T;\HH_{2-j},\HH_{-j})$
satisfying \eqref{a5}.

\subsubsection{Existence and uniqueness results}
For  $j=0,1,2$, we consider the Cauchy problem
\bel{s0}
\left\{ \begin{array}{l} -{\rm i} \psi' + H(t) \psi = f\ {\rm in}\ {\rm L}^2(0,T;\HH_{-j})  \\ \psi(0) = \psi_0, \end{array} \right.
\ee
where $\psi_0 \in \HH_{2-j}$ and $f \in {\rm L}^2(0,T;\HH_{-j})$.

Notice that the second line in \eqref{s0} makes sense since $\psi(0)$ is well defined in $\HH_{1-j}$ for every $j=0,1,2$. Indeed we have $W(0,T;\HH_{2-j},\HH_{-j}) \hookrightarrow {\rm C}^0([0,T];\HH_{1-j})$, $j=0,1,2$, as can be seen from \cite{DL5}[XVIII,\S1,(1.61)(iii)] for $j=0$, and from \cite{DL5}[XVIII,\S 1,Theorem 1] for $j=1$, the case
$j=2$ being a direct consequence of the imbedding $W(0,T;\HH_2,\HH_0) \hookrightarrow {\rm C}^0([0,T];\HH_{1})$ and the fact that $(1-\Delta)^{-1}$ maps $W(0,T;\HH_0,\HH_{-2})$ one-to-one onto $W(0,T;\HH_2,\HH_0)$.

Further, with reference to \cite{Eskin3}[Lemma 2.1] for $j=0$, and to \cite{LM}[Chap. 3, Theorem 10.1 \& Remark 10.2] (see also \cite{DL5}[XVIII, \S 7, Theorem 1 \& Remark 3]) for $j=1$, we recall the following existence and uniqueness result:
\begin{lemma}
\label{lm-r1}
Let $T$, $n$ and $\Omega$ be as in \S \ref{sec-wwaaf}, let $\ba \in {\rm C}^1([0,T];{\rm H}_{{\rm div}0}(\Omega;\R))$, and fix $j=0,1$. Then for all $\psi_0 \in \HH_{2-j}$ and $f \in W(0,T;\HH_0,\HH_{-j})$, there exists a unique solution $\psi  \in {\rm C}^0([0,T];\HH_{2-j}) \cap {\rm C}^1([0,T];\HH_{-j})$ to \eqref{s0}.
\end{lemma}
In light of Lemma \ref{lm-r1}, it is thus equivalent to solve \eqref{s0} with either $j=0$ or $j=1$ when $\psi_0 \in \HH_2$ and $f \in W(0,T;\HH_0)$. Moreover, due to the following uniqueness result, the same is true for either $j=1$ or $j=2$ provided $\psi_0 \in \HH_1$ and $f \in W(0,T;\HH_0,\HH_{-1})$:
\begin{lemma}
\label{lm-unique}
Let $T$, $n$ and $\Omega$ be the same as in \S \ref{sec-wwaaf}.
Then, for all $f \in {\rm L}^2(0,T;\HH_{-2})$ and $\psi_0 \in \HH_{-1}$, there exists at most one solution to the problem \eqref{s0} for $j=2$.
\end{lemma}
\begin{proof}
It is enough to prove that the only solution $\psi$ to the problem \eqref{s0} with $f=0$ and $\psi_0=0$ is identically zero.
Actually for all $\tau \in (0,T)$, the first line of \eqref{s0} with $f=0$ states that we have
$\langle \psi'(\tau) , v \rangle_{-2,2} = - {\rm i} \langle \psi(\tau) , H(\tau) v \rangle_0$
for any $v \in \HH_2$. In the particular case where $v=R(\tau) \psi(\tau)$, with $R(\tau):=(1+H(\tau))^{-1}$, this yields
\bel{z7}
\langle \psi'(\tau) , R(\tau) \psi(\tau) \rangle_{-2,2} =  {\rm i} ( \| R(\tau)^{1 \slash 2} \psi(\tau) \|_0^2
- \| \psi(\tau) \|_0^2).
\ee
In light of \eqref{r13}, $\tau \mapsto R(\tau) \varphi$ is differentiable in $(0,T)$ for each $\varphi \in \HH_0$, and we get $\frac{{\rm d}}{{\rm d} \tau} R(\tau) \varphi = R(\tau) H'(\tau) R(\tau) \varphi$ by standard computations. Therefore
$\tau \mapsto \langle R(\tau) \psi(\tau) , \varphi \rangle_0 = \langle  \psi(\tau) ,  R(\tau) \varphi \rangle_0$ is differentiable in $(0,T)$ as well, and we find that
\bea
\frac{{\rm d}}{{\rm d} \tau} \langle R(\tau) \psi(\tau) , \varphi \rangle_0
& = & \langle \psi'(\tau) , R(\tau) \varphi \rangle_{-2,2} + \langle \psi(\tau) , R(\tau) H'(\tau) R(\tau) \varphi \rangle_0 \nonumber \\
& = & \langle \psi'(\tau) , R(\tau) \varphi \rangle_{-2,2} + \langle R(\tau) H'(\tau) R(\tau) \psi(\tau) ,  \varphi \rangle_0. \label{z7b}
\eea
Recalling that $H(\tau)$ is extended into an operator mapping $\HH_0$ into $\HH_{-2}$, defined by
$\langle H(\tau) u , v \rangle_{-2,2} = \langle u , H(\tau) v \rangle_0$ for all $u \in \HH_0$ and $v \in \HH_2$,
the resolvent $R(\tau)$ can therefore be extended into an operator mapping $\HH_{-2}$ into $\HH_0$, according to the identity
$$ \langle R(\tau) u , v \rangle_0 :=  \langle u ,  R(\tau) v \rangle_{-2,2},\ u \in \HH_{-2},\ v \in \HH_0. $$
This, together with \eqref{z7b} proves that $(0,T) \ni \tau \mapsto R(\tau) \psi(\tau)$ is differentiable in $\HH_0$, with
$\frac{{\rm d}}{{\rm d} \tau} R(\tau) \psi(\tau)
= R(\tau) \psi'(\tau) + R(\tau) H'(\tau) R(\tau) \psi(\tau)$.
As a consequence,
$\tau \mapsto \| R(\tau)^{1 \slash 2} \psi(\tau) \|_0^2= \langle \psi(\tau) , R(\tau) \psi(\tau) \rangle_0$
is differentiable on $(0,T)$, and we have
$$ \frac{{\rm d}}{{\rm d} \tau} \| R(\tau)^{1 \slash 2} \psi(\tau) \|_0^2 = \langle \psi(\tau) , R(\tau) H'(\tau) R(\tau) \psi(\tau) \rangle_0
+ 2 \Pre{\langle \psi'(\tau), R(\tau) \psi(\tau) \rangle_{-2,2}}. $$
The second term in the r.h.s. of the above identity being zero according to \eqref{z7}, we end up getting that
\bel{z8}
\frac{{\rm d}}{{\rm d} \tau} \| R(\tau)^{1 \slash 2} \psi(\tau) \|_0^2 =
\langle R(\tau) \psi(\tau) , H'(\tau) R(\tau) \psi(\tau) \rangle_0.
\ee
Bearing in mind that $H'(\tau)=2 B_1(\tau)$ and using \eqref{r12c} we get that
\bel{z9}
\| H'(\tau) R(\tau)^{1 \slash 2} \| \leq 2 \mathcal{A}_1,\ \tau \in (0,T),
\ee
since
$\| B_1(\tau) R(\tau)^{1 \slash 2} u \|_0 = \| \ba'(t) . ( {\rm i} \nabla + \ba(\tau) ) R(\tau)^{1 \slash 2} u \|_0 \leq \mathcal{A}_1 \| ( {\rm i} \nabla + \ba(\tau) ) R(\tau)^{1 \slash 2} u \|_0$
for all $u \in \HH_0$, and
$$\| ( {\rm i} \nabla + \ba(\tau) ) R(\tau)^{1 \slash 2} u \|_0^2
= \| H(\tau)^{1 \slash 2} R(\tau)^{1 \slash 2} u \|_0^2 \leq \| u \|_0^2.$$
Putting \eqref{z8} and \eqref{z9} together, we obtain that
$\frac{{\rm d}}{{\rm d} \tau} \| R(\tau)^{1 \slash 2} \psi(\tau) \|_0^2 \leq 2 \mathcal{A}_1
\| R(\tau)^{1 \slash 2} \psi(\tau) \|_0^2$ for every $\tau \in (0,T)$,
hence
$$  \| R(t)^{1 \slash 2} \psi(t) \|_0^2 \leq 2 \mathcal{A}_1 \int_0^t \| R(\tau)^{1 \slash 2} \psi(\tau) \|_0^2 d \tau,\ t \in [0,T], $$
by integrating w.r.t. $\tau$ over $[0,t]$ and using the fact that $\psi(0)=0$.
Therefore $R(t)^{1 \slash 2} \psi(t)=0$ for all $t \in (0,T)$ by Gronwall inequality,
which yields $\psi=0$ and proves the result.
\end{proof}

In light of section \S \ref{sec-tdh} it is possible to differentiate (at least in a formal way) the Cauchy problem \eqref{s0} w.r.t. $t$. This is made precise with the following:
\begin{lemma}
\label{lm-r2}
Let $T$, $n$, $\Omega$ and $\ba$ be as in Lemma \ref{lm-r1},
let $f \in W(0,T;\HH_0)$, $\psi_0 \in \HH_2$, and let $\psi$ denote the ${\rm C}^0([0,T];\HH_2) \cap {\rm C}^1([0,T];\HH_0)$-solution to \eqref{s0}. Then $\psi'' \in {\rm L}^2(0,T;\HH_{-2})$, and $\psi'$ is solution to the system:
\bel{s1}
\left\{ \begin{array}{l} -{\rm i} \psi'' + H(t) \psi' = f' - H'(t) \psi\ {\rm in}\ {\rm L}^2(0,T;\HH_{-2}) \\ \psi'(0) = -{\rm i} ( H(0) \psi_0 - f(0) ). \end{array} \right.
\ee
\end{lemma}
\begin{proof}
The existence and uniqueness of $\psi \in {\rm C}^0([0,T];\HH_2) \cap {\rm C}^1([0,T];\HH_0)$ being guaranteed by Lemma \ref{lm-r1}, $j=0$, the mapping $t \mapsto \langle H(t) \psi(t) , v \rangle_0 = \langle \psi(t) , H(t) v \rangle_0$ is thus continuously differentiable in $[0,T]$ for every $v \in \HH_2$, and it holds true that:
$$ \frac{{\rm d}}{{\rm d} t} \langle H(t) \psi(t) , v \rangle_0  = \langle \psi'(t) , H(t) v \rangle_0 + \langle \psi(t) , H'(t) v \rangle_0
= \langle H(t) \psi'(t)  +  H'(t) \psi(t) , v \rangle_{-2,2}.
$$
Further we have $t \mapsto \langle f, v \rangle_0 \in W(0,T;\C)$ by assumption, hence the same is true for $t \mapsto \langle \psi'(t) , v \rangle_0$ from \eqref{s0}, and the result follows.
\end{proof}

\subsubsection{Uniform time boundedness of charge and energy revisited}
\label{sec-cce}
We establish in Lemma \ref{lm-ec} that the charge and the energy of the system driven by \eqref{s0} remain uniformly bounded in the course of time over $[0,T]$. These estimates, which are reminiscent of \cite{DL5}[Chap. XVIII, Formula (7.17)] or \cite{BP}[Lemma 7], are obtained here in the more general ${\rm L}^2(0,T;\HH_{-1})$-framework for \eqref{s0}. They are further extended to the cases of the first and second time derivatives of the above quantities in Proposition \ref{pr-ec}. This result is one of the main ingredients in the derivation of
the Carleman estimate of Proposition \ref{pr-Carly}.

\begin{lemma}
\label{lm-ec}
Let $T$, $n$, $\Omega$ and $\ba$ be as in Lemma \ref{lm-r1}.
Then, there is a constant $c_0>0$, depending only on $T$, $\mathcal{A}_0$, $\mathcal{A}_1$ and $\Omega$, such that for all $\psi_0 \in \HH_1$ and all $f \in W(0,T;\HH_0,\HH_{-1})$, the ${\rm C}^0([0,T];\HH_1) \cap {\rm C}^1([0,T];\HH_{-1})$-solution $\psi$ to \eqref{s0}, satisfies simultaneously:
\bel{ec0}
\| \psi(t) \|_0 \leq {\rm e}^{T \slash 2} ( \| \psi_0 \|_0 + \| f \|_{{\rm L}^2(0,T;\HH_0)} ),\ t \in [0,T],
\ee
and
\bel{ec1}
\| \psi(t) \|_1 \leq c_0 ( \| \psi_0 \|_1 + \| f \|_{W(0,T;\HH_0,\HH_{-1})}),\ t \in [0,T].
\ee
\end{lemma}
\begin{proof}
We have
$- {\rm i} \langle \psi'(t) , \psi(t) \rangle_0 + h(t;\psi(t),\psi(t)) = \langle f(t) , \psi(t) \rangle_0$ for all $t \in [0,T]$,
by multiplying \eqref{s0} by $\overline{\psi(t,x)}$ and integrating w.r.t. $x$ over $\Omega$. Selecting the
imaginary part in the obtained expression then yields
$\Pre{\langle \psi'(t) , \psi(t) \rangle_0} = - \Pim{\langle f(t), \psi(t) \rangle_0}$, or equivalently
$\frac{{\rm d}}{{\rm d} t} \| \psi(t) \|_0^2 = - 2 \Pim{\langle f(t), \psi(t) \rangle_0}$.
This entails
$$\| \psi(t) \|_0^2 \leq \| \psi_0 \|_0^2 + \| f \|_{{\rm L}^2(0,T;\HH_0)}^2 + \int_0^t \| \psi(s) \|_0^2 {\rm d} s,\ t \in [0,T],$$
and \eqref{ec0} follows immediately from this and Gronwall inequality.

We turn now to proving \eqref{ec1}. The set $W(0,T;\HH_0)$ being dense in $W(0,T;\HH_0,\HH_{-1})$, we consider a sequence $(f_n)_n$ of
$W(0,T;\HH_0)$ satisfying $\| f - f_n \|_{W(0,T;\HH_0,\HH_{-1})} \rightarrow 0$ as $n$ goes to infinity. Similarly we pick a sequence $(\psi_{n,0})_n$ of $\HH_2$ such that $\lim_{n \rightarrow \infty} \| \psi_{n,0} - \psi_0 \|_1 = 0$, and, for each $n \in \N$, we call $\psi_n$ the ${\rm C}^0([0,T];\HH_2) \cap {\rm C}^1([0,T];\HH_0)$-solution to \eqref{s0}, where $(f,\psi_0)$ is replaced by $(f_n,\psi_{n,0})$:
\bel{s0n}
\left\{ \begin{array}{l} -{\rm i} \psi_n' + H(t) \psi_n = f_n\ {\rm in}\ {\rm L}^2(0,T;\HH_{0})  \\ \psi_n(0) = \psi_{n,0}. \end{array} \right.
\ee
From \eqref{s0}, \eqref{ec0} and \eqref{s0n} then follows for every $n \in \N$ that
\bel{ec1b}
\| \psi(t) - \psi_n(t) \|_0 \leq {\rm e}^{T \slash 2} ( \| \psi_0 - \psi_{n,0} \|_0 + \| f - f_n \|_{{\rm L}^2(0,T;\HH_0)}),\ t \in [0,T].
\ee
Further, multiplying \eqref{s0n} by $\overline{\psi_n'(t,x)}$, integrating w.r.t. $x$ over $\Omega$, and selecting the real part of the deduced result, entails
$\Pre{\langle H(t) \psi_n(t) , \psi_n'(t) \rangle_0} = \Pre{\langle f_n(t) , \psi_n'(t) \rangle_0}$ for all
$t \in [0,T]$. Hence
$\frac{{\rm d}}{{\rm d} t} h(t;\psi_n(t),\psi_n(t)) = h'(t;\psi_n(t),\psi_n(t)) + 2 \Pre{\langle f_n(t) , \psi_n'(t) \rangle_0}$, which yields
\bea
& & h(t;\psi_n(t),\psi_n(t)) - h(0;\psi_{n,0},\psi_{n,0}) \nonumber \\
& = & \int_0^t h'(s;\psi_n(s),\psi_n(s)) {\rm d} s +
2 \Pre{\int_0^t \langle f_n(s) , \psi_n'(s) \rangle_0} {\rm d} s,\ t \in [0,T]. \label{ec2}
\eea
Moreover, as $\int_0^t \langle f_n(s) , \psi_n'(s) \rangle_0 {\rm d} s = \langle f_n(t) , \psi_n(t) \rangle_0 - \langle f_n(0) , \psi_{n,0} \rangle_0 - \int_0^t \langle f_n'(s) , \psi_n(s) \rangle_0 {\rm d} s$, we have
\bea
\left| \int_0^t \langle f_n(s) , \psi_n'(s) \rangle_0 {\rm d} s \right|
&\leq &\| f_n(t) \|_{-1} \| \psi_n(t) \|_1 +
\| f_n(0) \|_{-1} \| \psi_{n,0} \|_1 \nonumber \\
&&+ \int_0^t \| f_n'(s) \|_{-1} \| \psi_n(s) \|_1 {\rm d} s \nonumber \\
& \leq & c_1 (1+\epsilon^{-1}) \| f_n \|_{W(0,T;\HH_0,\HH_{-1})}^2 + \| \psi_{n,0} \|_1^2 \nonumber \\
&&+ \epsilon \| \psi_n(t) \|_1^2 +  \int_0^t \| \psi_n(s) \|_1^2 {\rm d} s, \label{ec2a}
\eea
for every $\epsilon \in (0,1)$, the constant $c_1>0$ depending only on $\Omega$.
Here we used the fact that $W(0,T;\HH_0,\HH_{-1}) \hookrightarrow {\rm C}^0([0,T];\HH_{-1})$ (see \cite{DL5}[Chap. XVIII \S 1, Formula (1.61)(iii)]).
Further, putting \eqref{ec0} and \eqref{ec2}-\eqref{ec2a} together, and taking into account that $h$ is $\HH_1$-coercive w.r.t. $\HH_0$, we may choose $\epsilon>0$ small enough so that we have
$$ \| \psi_n(t) \|_1^2 \leq c_2 \left( \| \psi_{n,0} \|_1^2 + \| f_n \|_{W(0,T;\HH_0,\HH_{-1})}^2  + \int_0^t \| \psi_n(s) \|_1^2 {\rm d} s \right),\ t \in [0,T], $$
for some constant $c_2>0$ depending only on $T$, $\mathcal{A}_0$, $\mathcal{A}_1$ and $\Omega$. Therefore we end up getting
\bel{ec3}
\| \psi_n(t) \|_1 \leq c_0 ( \| \psi_{n,0} \|_1^2 + \| f_n \|_{W(0,T;\HH_0,\HH_{-1})}^2 ),\ t \in [0,T],
\ee
from the Gronwall inequality, where $c_0$ fulfills the conditions of \eqref{ec1}.

From \eqref{s0n} and \eqref{ec3} then follows that $(\psi_n)_n$ is a bounded sequence in ${\rm L}^{\infty}(0,T;\HH_1)$, and $(\psi_n')_n$ is bounded in ${\rm L}^{\infty}(0,T;\HH_{-1})$. Moreover, for every $t \in [0,T]$,  $(\psi_n(t))_n$ converges strongly to $\psi(t)$ in $\HH_0$ by \eqref{ec1b}, hence weakly in $\HH_1$ according to \cite{Caz}[Proposition 1.3.14(i)]. Therefore we have $\| \psi(t) \|_1 \leq \liminf_{n \rightarrow \infty} \| \psi_n(t) \|_1$ from the weak lower semicontinuity of the ${\rm H}^1$-norm (see e.g. \cite{Caz}[Remark 1.3.1(iii)]), so \eqref{ec1} follows from this, \eqref{ec3}, and the asymptotic behavior of $(f_n)_n$ and $(\psi_{n,0})_n$.
\end{proof}

We now prove the main result of this section:
\begin{proposition}
\label{pr-ec}
Let $T$, $n$ and $\Omega$ be as in \S \ref{sec-wwaaf}, and let $\ba \in {\rm C}^2([0,T];{\rm H}_{{\rm div}0}(\Omega;\R)$.
Assume that $f \in W^2(0,T;\HH_0,\HH_{-1})$, $\psi_0 \in \HH_2$, $f(0)-H(0) \psi_0 \in \HH_2$ and $f'(0)-H'(0) \psi_0 + {\rm i} (H(0) f(0) - H(0)^2 \psi_0) \in \HH_1$. Then there is a constant $c>0$ depending only on $T$, $\mathcal{A}_0$, $\mathcal{A}_1$, $\mathcal{A}_2$ and $\Omega$, such that the solution $\psi$ to \eqref{s0} satisfies:
\bel{ec}
\left\| (\partial^j \psi \slash \partial t^j)(t)  \right\|_1 \leq c \sum_{k=0}^j \left[ \| \Delta^k \psi_0 \|_1 + \left\| \partial^k f \slash \partial t^k \right\|_{W(0,T;\HH_0,\HH_{-1})} \right],\ j=0,1,2,\ t \in [0,T].
\ee
\end{proposition}
\begin{proof}
The case $j=0$ following directly from Lemma \ref{lm-ec} we first prove \eqref{ec} for $j=1$.
Taking into account that $\psi_0 \in \HH_2$ and $f \in W(0,T;\HH_0)$, we have $\psi \in {\rm C}^0([0,T];\HH_2) \cap {\rm C}^1([0,T];\HH_0)$ from Lemma \ref{lm-r1}, and $\psi'$ is solution to the system
\bel{ec4}
\left\{ \begin{array}{l} -{\rm i} \psi'' + H(t) \psi' = f_1 := f' -H'(t) \psi\ {\rm in}\ {\rm L}^2(0,T;\HH_{-2})  \\ \psi'(0) = {\rm i} (f(0) - H(0) \psi_0), \end{array} \right.
\ee
by Lemma \ref{lm-r2}. Further, $f_1 \in W(0,T;\HH_0,\HH_{-1})$ since $f' \in W(0,T;\HH_0,\HH_{-1})$, and we have
\bel{ec5}
\| f_1 \|_{{\rm L}^2(0,T;\HH_0)} \leq \| f' \|_{{\rm L}^2(0,T;\HH_0)} + \ell_1 T  \| \psi \|_{{\rm L}^{\infty}(0,T;\HH_1)} \leq c_3 ( \| \psi_0 \|_1 + \| f' \|_{{\rm L}^2(0,T;\HH_0)} ),
\ee
according to \eqref{r14} and \eqref{ec} with $j=0$, for some constant $c_3>0$ depending only on $T$, $\mathcal{A}_0$, $\mathcal{A}_1$ and $\Omega$. This combined with \eqref{ec0} entails
\bel{ec6}
\| \psi'(t) \|_0 \leq c_4 ( \| \psi_0 \|_1 + \| \Delta \psi_0 \|_0 ),\ t \in [0,T],
\ee
where $c_4>0$ depends on $T$, $\mathcal{A}_0$, $\mathcal{A}_1$ and $\Omega$ as well.
Moreover, bearing in mind that $f_1'(t)=f''(t)-H'(t) \psi'(t) - H''(t) \psi(t)$, we find that
$$
\| f_1' \|_{{\rm L}^2(0,T;\HH_{-1})} \leq \| f'' \|_{{\rm L}^2(0,T;\HH_{-1})} + \ell_1 T \| \psi' \|_{{\rm L}^{\infty}(0,T;\HH_0)} + \ell_2 T \| \psi \|_{{\rm L}^{\infty}(0,T;\HH_0)},
$$
which, together with \eqref{ec1} and \eqref{ec5}-\eqref{ec6}, proves \eqref{ec} for $j=1$.

Further, in light of \eqref{ec4}, $\Delta \psi_0 \in \HH_1$ and $f_1 \in W(0,T;\HH_0,\HH_{-1})$, we know from Lemma \ref{lm-r1} that $\psi \in {\rm C}^1([0,T];\HH_1)$. This shows that $f_1 \in W(0,T;\HH_0)$ and, $\Delta \psi_0$ being taken in $\HH_2$, Lemma \ref{lm-r2} guarantees that $\psi''$ is solution to the system
\bel{ec8}
\left\{ \begin{array}{l} -{\rm i} \psi''' + H(t) \psi'' = f_2 := f''-2H'(t) \psi' - H''(t) \psi\ {\rm in}\ {\rm L}^2(0,T;\HH_{-2})  \\ \psi''(0) = {\rm i} (f'(0) - H'(0) \psi_0 ) - H(0) f(0) + H(0)^2 \psi_0. \end{array} \right.
\ee
Finally applying Lemma \ref{lm-r1} to \eqref{ec4} once more we see that
$\psi \in {\rm C}^1([0,T];\HH_2) \cap {\rm C}^2([0,T];\HH_0)$. Therefore $f_2 \in W(0,T;\HH_0,\HH_{-1})$, and the case $j=2$ follows from \eqref{ec8}, by arguing in the same way as in the derivation of \eqref{ec} with $j=1$, from \eqref{ec5}.
\end{proof}

\section{Stability inequality}
\label{sec-si}
\setcounter{equation}{0}
Let $\Omega$, $\ba_0$ and $M$ be the same as in \S \ref{intro-mr}. Then we consider the time depend magnetic Hamiltonian $H_{\ba}(t)=H(t)$, $t \in [0,T]$, defined in \S \ref{sec-mso} and associated to a magnetic potential of the form
$\ba(t,x):=\chi(t) \ba(x)$, where
\bel{defchi}
\ba \in {\bf A}(\ba_0,M)\ {\rm and}\ \chi \in {\rm C}^3([0,T];\R)\ {\rm is\ such\ that}\ \chi(0) = 0\ {\rm and}\ \chi'(0) \neq 0.
\ee
Notice in this case that the operators $B_{j,\ba}(t)=B_j(t)$, $j=1,2,3$, defined in \S \ref{sec-tdh}, have the following expression: $B_{j,\ba}(t) = \chi^{(j)}(t) \ba . ( {\rm i} \nabla + \chi(t) \ba )$, where $\chi^{(j)}(t):=\frac{{\rm d}^j \chi}{{\rm d} t^j}(t)$.

\subsection{The inverse problem}
\label{sec-ip}
In this section we discuss the inverse problem of determining the magnetic potential $\ba \in {\bf A}(\ba_0,M)$ from the measurement on a part $\Gamma^+$ of the boundary $\Gamma$, of Neumann data of the solution $u$ (more precisely the first and second order time derivatives of the flux of $u$) to the problem \eqref{ia1},
where $u_0$ is a suitable real-valued function on $\Omega$, satisfying:
\bel{cuz}
\Delta^k u_0 \in \HH_2,\ k=0,1,2.
\ee
In light of \eqref{ia1} and Lemma \ref{lm-r1}, and since $u_0 \in \HH_2$, we have $u \in {\rm C}^0([0,T];\HH_2) \cap {\rm C}^1([0,T];\HH_0)$, and $u'$ is solution to the system
\bel{ia1b}
\left\{  \begin{array}{ll} -{\rm i} u''(t,x) + H_{\ba}(t) u'(t,x) = f_1(t,x) := -H_{\ba}'(t) u(t,x), & (t,x) \in Q_T^+ \\ u'(t,x) = 0, & (t,x) \in \Sigma_T^+ \\ u'(0,x) = {\rm i} \Delta u_0(x), & x \in \Omega, \end{array} \right.
\ee
according to Lemma \ref{lm-r2}. Further, using that $f_1 \in W(0,T;\HH_0,\HH_{-1})$ and $\Delta u_0 \in \HH_1$, it follows from \eqref{ia1b} and Lemma \ref{lm-r1} that $u' \in {\rm C}^1(0,T;\HH_1) \cap {\rm C}^0([0,T];\HH_{-1})$, which, in turn, proves that $f_1 \in W(0,T;\HH_0)$. Thus, taking into account that $\Delta u_0 \in \HH_2$, we obtain that $u \in {\rm C}^1([0,T];\HH_2) \cap {\rm C}^2([0,T];\HH_0)$ from \eqref{ia1b} and Lemma \ref{lm-r1}. Moreover, by Lemma \ref{lm-r2}, $u''$ is solution to the problem
\bel{ia1c}
\left\{  \begin{array}{ll} -{\rm i} u'''(t,x) + H_{\ba}(t) u''(t,x) = f_2(t,x), & (t,x) \in Q_T^+  \\ u''(t,x) = 0, & (t,x) \in \Sigma_T^+ \\ u''(0,x) = - \Delta^2 u_0(x), & x \in \Omega, \end{array} \right.
\ee
with $f_2 :=-H_{\ba}''(t) u- 2 H_{\ba}'(t) u'$.
Finally, by substituting $u''$ (resp. $f_2$, $\Delta^2 u_0$ and \eqref{ia1c}) for $u'$ (resp. $f_1$, $\Delta u_0$ and \eqref{ia1b}) in the above reasoning, we end up getting that $u \in {\rm C}^2([0,T];\HH_2) \cap {\rm C}^3([0,T];\HH_0)$.

Further, $\ba$ and $\tilde{\ba}$ being taken in ${\bf A}(\ba_0,M)$, let $u$ and $\tilde{u}$ be the respective ${\rm C}^2([0,T];\HH_2) \cap {\rm C}^3([0,T];\HH_0)$-solutions to \eqref{ia1} and the system
\bel{ia2} \left\{  \begin{array}{ll} -{\rm i} \tilde{u}'(t,x) + H_{\tilde{\ba}}(t) \tilde{u}(t,x) = 0, & (t,x) \in Q_T^+ \\ \tilde{u}(t,x) = 0, & (t,x) \in \Sigma_T^+  \\ \tilde{u}(0,x) = u_0(x), & x \in \Omega. \end{array} \right.
\ee
Then, bearing in mind that $\nabla . \ba = \nabla . \tilde{\ba}=0$, we find using basic computations that
$v:=u- \tilde{u} \in \HH_1$ satisfies
\bel{ia3}
\left\{  \begin{array}{ll} -{\rm i} v'(t,x) + H_{\ba}(t) v(t,x) = f(t,x), & (t,x) \in Q_T^+  \\ v(t,x) = 0, & (t,x) \in \Sigma_T^+ \\ v(0,x) = 0, & x \in \Omega, \end{array} \right.
\ee
where
\bel{ia4}
f:= \chi (\tilde{\ba}-\ba) . \left( 2 {\rm i} \nabla + \chi (\ba+\tilde{\ba})  \right) \tilde{u} \in W^2(0,T;\HH_0,\HH_{-1}).
\ee
Since $f \in W(0,T;\HH_0)$, we have $v \in {\rm C}^0([0,T];\HH_2) \cap {\rm C}^1([0,T];\HH_0)$ by Lemma \ref{lm-r1}, and $w:= v'$ is thus solution to
\bel{ia6}
\left\{  \begin{array}{ll} -{\rm i} w'(t,x) + H_{\ba}(t) w(t,x) = g(t,x) := f'(t,x) - H_{\ba}'(t) v(t,x), & (t,x) \in Q_T^+ \\ w(t,x) = 0, & (t,x) \in \Sigma_T^+ \\ w(0,x) = 0, & x \in \Omega, \end{array}
\right.
\ee
according to Lemma \ref{lm-r2}. As $g \in W(0,T;\HH_0,\HH_{-1})$,  \eqref{ia6} and Lemma \ref{lm-r1} yield
$w \in {\rm C}^0([0,T];\HH_1)$, which, in turn, implies that $g$ actually belongs to $W(0,T;\HH_0)$. Therefore, we get that $w \in {\rm C}^0([0,T];\HH_2) \cap {\rm C}^1([0,T];\HH_0)$ and $y:=w'$ is solution to
\bel{ia9}
\left\{  \begin{array}{ll} -{\rm i} y'(t,x) + H_{\ba}(t) y(t,x) = q(t,x) := g'(t,x) - H_{\ba}'(t) w(t,x), & (t,x) \in Q_T^+  \\ y(t,x) = 0, & (t,x) \in \Sigma_T^+ \\ y(0,x) = -2 \chi'(0) (\tilde{\ba}(x)-\ba(x)) . \nabla  u_0(x), & x \in \Omega, \end{array} \right.
\ee
by arguing as before. Finally, taking into account that $(\tilde{\ba}-\ba) . \nabla  u_0 \in \HH_1$ (since $\ba$ and $\tilde{\ba}$ coincide in a neighborhood $\gamma_0$ of $\Gamma$) and $q \in W(0,T;\HH_0,\HH_{-1})$, we find out from \eqref{ia9} and Lemma \ref{lm-r1}, that $y$ is uniquely defined in ${\rm C}^0([0,T];\HH_1) \cap {\rm C}^1([0,T];\HH_{-1})$.

Though, in this framework, additional technical difficulties arise from the presence of first order spatial differential terms in the expression of $H_{\ba}(t)$, the method used in \S \ref{sec-li} to establish a stability estimate on $\ba$ is inspired from \cite{ImaYama01}. This method of symmetrization already used in \cite{BP} and \cite{CCG} for Schr\"odinger systems allows us to center the problem around the initial condition which is a given data in our problem and avoid the use of another data at  time $\theta >0$. This method  preliminarily requires that a Carleman inequality for the solution $y$ to \eqref{ia9} be established in $Q_T :=(-T,T) \times \Omega$. To this purpose we first extend $\chi$ in an odd function on $[-T,T]$, and introduce the corresponding operators $H_{\ba}(t)$ and $B_{\ba}(t)$ for every $t \in [-T,T]$. Then we extend $u$ (resp. $\tilde{u}$) on $Q_T^- :=(-T,0) \times \Omega$ by setting $u(t,x)=\overline{u(-t,x)}$  (resp. $\overline{\tilde{u}}(t,x)=\overline{\tilde{u}(-t,x)}$) for $(t,x) \in Q_T^-$, in such a way that $u$ (resp. $\tilde{u}$) satisfies \eqref{ia1} (resp. \eqref{ia2}) in $Q_T$. Since $u_0$ is real-valued by assumption, the mappings $t \mapsto u(t,x)$ and $t \mapsto \tilde{u}(t,x)$ are actually both continuous at $t=0$ for a.e. $x \in \Omega$.
Further, by extending the above definitions of $v$ and $f$ in $Q_T^-$, we get that $v(t,x) = \overline{v(-t,x)}$ and $f(t,x) = \overline{f(-t,x)}$ for all $(t,x) \in Q_T^-$. Hence $v$ is solution to \eqref{ia3} in $Q_T$, and, due to the initial condition $v(0,x)=0$, $t \mapsto v(t,x)$ is actually continuous at $t=0$ for a.e. $x \in \Omega$. Similarly we may define $w$ and $g$ on $Q_T^-$ from the extended definitions of $v$ and $f$, as $w(t,x)=-\overline{w(-t,x)}$ and $g(t,x)=-\overline{g(-t,x)}$ for $(t,x) \in Q_T^-$. This combined with the initial condition $w(0,x)=0$ guarantees that $t \mapsto w(t,x)$ is continuous at $t=0$ for a.e. $x \in \Omega$, and it is easy to check that $w$ is solution to \eqref{ia6} in $Q_T$. Finally, by setting $q(t,x) := \overline{q(-t,x)}$ and $y(t,x):=\overline{y(-t,x)}$ for $(t,x) \in Q_T^-$, we end up getting that $y$ is solution to \eqref{ia9} in $Q_T$. Moreover, in light of the identity $y(0,x)= -2 \chi'(0) (\tilde{\ba}(x)-\ba(x)) .  \nabla u_0(x)$ and the condition $u_0(x) \in \R$ imposed on $u_0$, the mapping $t \mapsto y(t,x)$ turns out to be continuous at $t=0$.

\subsection{Carleman estimate}
\label{sec-ce}

In this section we prove a Carleman inequality for the solution $y$ to the system \eqref{ia9}, extended to $Q_T$.
This is a powerful tool in the area of inverse problems which was introduced by A. L. Bugkheim and M. V. Klibanov in \cite{BK:81} (see also e.g. \cite{B:99} and \cite{K:92}). So far, the method defined in \cite{BK:81} is the only one enabling to prove uniqueness and stability results for inverse problems with finite measurement data in the multi-dimensional spatial case (i.e. $n \ge 2$).

Let us first recall some useful result borrowed from \cite{BP}, establishing a global Carleman estimate for any (sufficiently smooth) function $q$, defined in $Q_T$, which vanishes on $\Sigma_T:=(-T,T) \times \partial \Omega$, and the Schr\"odinger operator $L$ acting in $({\rm C}_0^{\infty})'(Q_T)$, as,
\bel{H}
L :={\rm i} \partial_t + \Delta.
\ee
This preliminarily requires that we introduce an open subset $\Gamma^+$ of $\Gamma$, together with some function $\tilde{\beta} \in {\rm C}^4(\overline{\Omega};\R_+)$, satisfying the following conditions:
\begin{assumption}
\label{funct-beta}
\begin{enumerate}[(a)]
\item $\exists C_0>0$ such that we have $|\nabla \tilde{\beta}(x)| \geq C_0$ for all $x \in \Omega$;
\item ${\partial}_{\nu} {\widetilde{\beta}}(\sigma) := \nabla {\tilde{\beta}}(\sigma). \nu (\sigma) \leq 0$ for all $\sigma \in \Gamma^- := \Gamma \backslash \Gamma^+$;
\item $\exists \Lambda_1>0$, $\exists \epsilon>0$  such that we have $\lambda |\nabla  \tilde{\beta}(x) . \zeta|^2 + D^2 \tilde{\beta} (\zeta, \bar{\zeta}) \geq \epsilon  |\zeta|^2$ for all $\zeta \in  \R^n$ and $\lambda  > \Lambda_1$, where $D^2 \tilde{\beta}=\left( \frac{\partial^2 \tilde{\beta}}{\partial x_i \partial x_j} \right)_{1 \leq i,j \leq n}$, and $D^2 \tilde{\beta} (\zeta, \bar{\zeta})$ denotes the $\C^n$-scalar product of $D^2 \tilde{\beta} \zeta$ with $\zeta$.
\end{enumerate}
\end{assumption}
Notice from Assumption \ref{funct-beta}(c) that $\tilde{\beta}$ is pseudo-convex with respect to the operator $-\Delta$. We refer to \cite{BP}[Formula (6)] for actual examples of an open subset $\Gamma^+$ and a weight $\tilde{\beta}$, fulfilling Assumption \ref{funct-beta}.

Further we put
\bel{defbeta}
\beta:= \widetilde{\beta}+K,\ {\rm where}\ K:= m \|\tilde{\beta}\|_{\infty}\ {\rm for\ some}\ m>1,
\ee
and then define for every $\lambda>0$ the two following weight functions:
\bel{defphieta}
\varphi(t,x)=\frac{{\rm e}^{\lambda  \beta(x)}}{(T+t)(T-t)}\ {\rm and}\ \eta(t,x)=\frac{{\rm e}^{2\lambda K} -{\rm e}^{\lambda  \beta(x)}}{(T+t)(T-t)},\ (t,x) \in Q_T.
\ee
Finally, following the idea of \cite{BP}[Proposition 3], we introduce two operators acting in $({\rm C}_0^{\infty})'(Q_T)$,
\bel{M1}
M_1 : = {\rm i} \partial_t + \Delta  + s^2 |\nabla \eta |^2\ {\rm and}\ M_2 : = {\rm i} s \eta' + 2 s \nabla \eta . \nabla  + s (\Delta \eta),
\ee
in such a way that $M_1+M_2={\rm e}^{-s \eta} L{\rm e}^{s \eta}$, where $L$ is given by \eqref{H}. Then, by arguing in the exact same way as in the derivation of \cite{BP}[Proposition 1], we obtain the:

\begin{proposition}
\label{pr-Carl}
Let $T$, $n$ and $\Omega$ be as in \S \ref{sec-wwaaf}, let $\beta$ be given by \eqref{defbeta}, where $\Gamma^+ \subset \Gamma$ and $\tilde{\beta} \in {\rm C}^4(\overline{\Omega};\R_+)$ fulfill Assumption \ref{funct-beta}, let $\varphi$ and $\eta$ be as in \eqref{defphieta}, and let
$L$, $M_1$ and $M_2$ be defined by \eqref{H}-\eqref{M1}. Then there exist three constants $\lambda_0> 0$, $s_0>0$, and $C_0=C_0(T,\Omega, \Gamma^+,\lambda_0,s_0)>0$, such that we have
\bel{Carl}
I(q) \leq C_0 \left( s \lambda \int_{-T}^T \int_{\Gamma^+} {\rm e}^{-2s \eta(t,\sigma)} \varphi(t,\sigma) \partial_{\nu} \beta(\sigma) | \partial_{\nu} q(t,\sigma) |^2 {\rm d} \sigma {\rm d} t + \| {\rm e}^{-s \eta} L q \|_{{\rm L}^2(Q_T)}^2 \right),
\ee
for all $\lambda \geq \lambda_0$, all $s \geq s_0$, and all $q \in {\rm L}^2(-T,T;\HH_1)$ satisfying $L q \in {\rm L}^2(Q_T)$ and $\partial_{\nu} q \in  {\rm L}^2(-T,T;{\rm L}^2(\Gamma))$, where
\bel{defiq}
I(q) := s^{3} \lambda^{4} \| {\rm e}^{-s \eta} \varphi^{3 \slash 2} q \|_{{\rm L}^2(Q_T)}^2 +
s \lambda \| {\rm e}^{-s \eta} \varphi^{1 \slash 2} | \nabla q | \|_{{\rm L}^2(Q_T)}^2 + \sum_{j=1,2} \| M_j e^{-s \eta}q \|_{{\rm L}^2(Q_T)}^2.
\ee
\end{proposition}

We turn now to proving a Carleman estimate for the solution $y$ to the system \eqref{ia9} in $Q_T$, i.e.
\bel{ia10}
\left\{  \begin{array}{ll} -{\rm i} y'(t,x) + H_{\ba}(t) y(t,x) = q(t,x), & (t,x) \in Q_T  \\ y(t,x) = 0, & (t,x) \in \Sigma_T := (-T,T) \times \Gamma \\ y(0,x) = -2 \chi'(0) (\tilde{\ba}(x)-\ba(x)) . \nabla  u_0(x), & x \in \Omega, \end{array} \right.
\ee
with the aid of Proposition \ref{pr-Carl}.
\begin{proposition}
\label{pr-Carly}
Let $n$ and $\Omega$ be as in Lemma \ref{lm-r1}, let $\beta$, $\varphi$ and $\eta$ be the same as in Proposition \ref{pr-Carl}, let $\chi$ fulfill \eqref{defchi}, let $u_0$ satisfy \eqref{cuz}, let $\ba$ and $\tilde{\ba}$ belong to ${\bf A}(\ba_0,M)$, and let $I(y)$ be defined by \eqref{defiq}, where $y$ denotes the solution to \eqref{ia10}.
Then there exist three constant $\lambda_0> 0$, $s_0>0$, and $C_1=C_1(T, M, \Omega, \Gamma^+,\lambda_0,s_0,u_0)>0$, such that we have
\beas
I(y) & \leq & C_1 \left( s \lambda \sum_{\rho=y,w} \int_{-T}^T \int_{\Gamma^+} {\rm e}^{-2s \eta(t,\sigma)} \varphi(t,\sigma) \partial_{\nu} \beta(\sigma)  | \partial_{\nu} \rho(t,\sigma) |^2 {\rm d} \sigma {\rm d} t \right. \\
& & \ \ \ \ \ \ \ \left. + \| {\rm e}^{-s \eta}
(\tilde{\ba}-\ba) \|_{{\rm L}^2(Q_T)^n}^2 \right),
\eeas
for any $\lambda \geq \lambda_0$ and any $s \geq s_0$.
\end{proposition}
\begin{proof}
In light of \eqref{H} we see that the first equation in \eqref{ia6} (extended to $Q_T$) reads
$L w = 2 \chi \ba . ( {\rm i}  \nabla +\chi \ba ) w - g$, the r.h.s. of this equality being in ${\rm L}^2(Q_T)$.
From this and Proposition \ref{pr-Carl} then follows that
\bea
I(w) \leq C_0 \left( s \lambda \int_{-T}^T \int_{\Gamma^+} {\rm e}^{-2s \eta(t,\sigma)} \varphi(t,\sigma) |\partial_{\nu} w(t,\sigma)|^2\ \partial_{\nu} \beta(t,\sigma) \ {\rm d} \sigma {\rm d} t \right. \nonumber  \\
\left. \ \ \ \ + \| {\rm e}^{-s \eta} f' \|_{{\rm L}^2(Q_T)}^2+ \sum_{\rho=v,w} \| {\rm e}^{-s \eta} (|\rho|^2 + | \nabla \rho |^2)^{1\slash 2} \|_{{\rm L}^2(Q_T)}^2 \right), \label{carl1}
\eea
provided $\lambda$ and $s$ are taken sufficiently large.
Similarly, we deduce from \eqref{ia10} that
\bea
I(y) \leq C_0 \left( s \lambda \int_{-T}^T \int_{\Gamma^+} {\rm e}^{-2s \eta(t,\sigma)} \varphi(t,\sigma) |\partial_{\nu} y(t,\sigma)|^2\ \partial_{\nu} \beta(t,\sigma) \ {\rm d} \sigma {\rm d} t \right. \nonumber  \\
\left. \ \ \ \ \ + \| {\rm e}^{-s \eta} f'' \|_{{\rm L}^2(Q_T)}^2+ \sum_{\rho=v,w,y} \| {\rm e}^{-s \eta} (|\rho|^2 + | \nabla \rho |^2)^{1\slash 2} \|_{{\rm L}^2(Q_T)}^2 \right), \label{carl2}
\eea
the constant $C_0$ being the same as in \eqref{carl1}.

Further, by applying Lemma \ref{lm-klibanov}, whose proof is postponed to the end of \S \ref{sec-ce}, successively to $v(x,t) = \int^t_{0} w(\xi,x) {\rm d} \xi$ and $\nabla v(x,t) = \int^t_{0} \nabla w(\xi,x) {\rm d} \xi$, we find out for all $\lambda \geq \lambda_0$ and all $s >0$, that
$$
\| {\rm e}^{-s \eta}\  (|v|^2 + |\nabla v|^2)^{1 \slash 2} \|_{{\rm L}^2(Q_T)}^2 \leq
\frac{\kappa}{s} \| {\rm e}^{-s \eta}\  (|w|^2 + |\nabla w|^2)^{1 \slash 2} \|_{{\rm L}^2(Q_T)}^2,
$$
where the constant $\kappa=\kappa(T,\lambda_0) >0$ depends only on $T$ and $\lambda_0$. Therefore, upon choosing $s$ sufficiently large and eventually substituting $2 C_0$ for $C_0$, we may actually remove $v$ from the sum in the r.h.s. of both \eqref{carl1} and \eqref{carl2}. Moreover $\| {\rm e}^{-s \eta} (|\rho|^2 + | \nabla \rho |^2)^{1\slash 2} \|_{{\rm L}^2(Q_T)}^2$, for $\rho=w,y$, being made arbitrarily small w.r.t. $I(\rho)$, by taking $s$ large enough, according to \eqref{defbeta}-\eqref{defphieta} and \eqref{defiq}, we may rewrite \eqref{carl1} and \eqref{carl2} as, respectively,
$$
I(w) \leq C_2 \left( s \lambda \int_{-T}^T \int_{\Gamma^+} {\rm e}^{-2s \eta(t,\sigma)} \varphi(t,\sigma) |\partial_{\nu} w(t,\sigma)|^2\ \partial_{\nu} \beta(t,\sigma) \ {\rm d} \sigma {\rm d} t + \| {\rm e}^{-s \eta} f' \|_{{\rm L}^2(Q_T)}^2 \right),
$$
and
\beas
I(y) & \leq & C_2 \left( s \lambda \int_{-T}^T \int_{\Gamma^+} {\rm e}^{-2s \eta(t,\sigma)} \varphi(t,\sigma) |\partial_{\nu} y(t,\sigma)|^2\ \partial_{\nu} \beta(t,\sigma) \ {\rm d} \sigma {\rm d} t  \right. \\
& & \ \ \ \ \ \ \left. + \| {\rm e}^{-s \eta} f'' \|_{{\rm L}^2(Q_T)}^2 + I(w) \right).
\eeas
Finally the result follows from this, the two following expressions,
$$f'=(\tilde{\ba}-\ba) . [ \chi  (\tilde{\ba}+\ba) (2\chi'\tilde{u} + \chi \tilde{u}') + 2 {\rm i}( \chi'  \nabla \tilde{u} + \chi \nabla \tilde{u}')],$$
and
$$f''=(\tilde{\ba}-\ba) . [ 2 {\rm i} ( \chi'' \nabla \tilde{u} + 2 \chi' \nabla \tilde{u}' + \chi \nabla \tilde{u}'' ) + (\tilde{\ba}+\ba) ( 2 \chi'^2 + \chi \chi'') \tilde{u} + 4 \chi \chi' \tilde{u}' + \chi^2 \tilde{u}'' ], $$
arising from \eqref{ia4} by direct calculation, and Proposition \ref{pr-ec} (in the particular case where the source term $f=0$).
\end{proof}

To complete the proof of Proposition \ref{pr-Carly}, it remains to establish Lemma \ref{lm-klibanov}. Its proof, although very similar to the one of \cite{KT:04}[Lemma 2.1] or \cite{KT}[Lemma 3.1.1], is detailed below for the convenience of the reader.

\begin{lemma}
\label{lm-klibanov}
Let $T$, $n$ and $\Omega$ be as in \S \ref{sec-wwaaf}, and let $\eta$ be defined by \eqref{defphieta}.
Then for all $\lambda_0>0$, there exists a constant $\kappa=\kappa(T,\lambda_0)>0$ depending only on $T$ and $\lambda_0$, such that we have
$$ \int_{-T}^T \int_{\Omega} {\rm e}^{-2s \eta(t,x)} \left| \int_0^t p(\xi,x) {\rm d} \xi \right|^2 {\rm d} x {\rm d} t \leq \frac{\kappa}{s} \| {\rm e}^{-s \eta} p \|_{{\rm L}^2(Q_T)}^2, $$
for every $p \in {\rm L}^2(Q_T)$, all $\lambda \geq \lambda_0$ and all $s>0$.
\end{lemma}
\begin{proof}
In light of \eqref{defbeta}-\eqref{defphieta}, it holds true that
\bel{kli1}
\partial_t \eta(t,x) = \frac{2t \alpha(x)}{(T+t)^2 (T-t)^2}\ {\rm with}\ \alpha(x)={\rm e}^{2 \lambda K} - {\rm e}^{\lambda \beta(x)} \geq \alpha_0 := {\rm e}^{2 \lambda_0 K} - {\rm e}^{\lambda_0 K \slash m}>0,
\ee
for every $(t,x) \in Q_T$ and $\lambda \geq \lambda_0$. Further, for all $\delta \in (0,T)$, the integral $I_{\delta} := \int_{-T+\delta}^{T-\delta} \int_{\Omega} {\rm e}^{-2s \eta(t,x)} \left| \int_0^t p(\xi,x) {\rm d} \xi \right|^2 {\rm d} x {\rm d} t$ satisfying
$$
I_{\delta}  \leq \int_{-T+\delta}^{T+\delta} \int_{\Omega} {\rm e}^{-2s \eta(t,x)} t  \left( \int_0^t | p(\xi,x) |^2 {\rm d} \xi \right) {\rm d} x {\rm d} t,
$$
from the Cauchy-Schwarz inequality, we get from \eqref{kli1} that
\bel{kli2}
I_{\delta} \leq \frac{2 T^4}{\alpha_0}  \int_{\Omega} \int_{-T+\delta}^{T-\delta}
{\rm e}^{-2s \eta(t,x)} \partial_t \eta(t,x)  \left( \int_0^t | p(\xi,x) |^2 {\rm d} \xi \right) {\rm d} t {\rm d} x.
\ee
Moreover, for a.e $x \in \Omega$, an integration by parts show that
\beas
& & \int_{-T+\delta}^{T-\delta} {\rm e}^{-2s \eta(t,x)} \partial_t \eta(t,x)  \left( \int_0^t | p(\xi,x) |^2 {\rm d} \xi \right) {\rm d} t \\
& = & \frac{1}{2 s} \left(  \int_{-T+\delta}^{T-\delta} {\rm e}^{-2s \eta(t,x)} | p(t,x) |^2  {\rm d} t  - \tilde{p}(T-\delta,x) + \tilde{p}(-T+\delta,x) \right),
\eeas
where $\tilde{p}(t,x):= {\rm e}^{-2s \eta(t,x)} \left( \int_0^{t} | p(\xi,x) |^2 {\rm d} \xi \right)$.
This together with \eqref{kli2} yields
\bea
I_{\delta} & \leq & \frac{T^4}{\alpha_0 s} \left(  \int_{-T+\delta}^{T-\delta} \int_{\Omega}
{\rm e}^{-2s \eta(t,x)} | p(t,x) |^2   {\rm d} x {\rm d} t \right. \nonumber \\
& & \ \ \ \ \ \ \ \ \ \left.- 2(T-\delta) \int_{\Omega} (\tilde{p}(T-\delta,x) + \tilde{p}(-T+\delta,x) ) {\rm d} x \right).
\label{kli3}
\eea
Since
$\int_{\Omega} \tilde{p}(\pm(T-\delta),x) {\rm d} x \leq {\rm e}^{-2 s \alpha_0 \slash (\delta (2T-\delta))} \| p \|_{{\rm L}^2(-T,T;\Omega)}^2$ by \eqref{defphieta} and \eqref{kli1}, we get that
$$\lim_{\delta \downarrow 0} (T-\delta) \int_{\Omega} (\tilde{p}(T-\delta,x) + \tilde{p}(-T+\delta,x) ) {\rm d} x=0,$$
and the result follows from this, by taking the limit as $\delta \downarrow 0$ in \eqref{kli3}.
\end{proof}

\subsection{Global Lipschitz stability inequality: proof of Theorem \ref{thm-stab}}
\label{sec-li}
In this section we establish the Lipschitz-type stability inequality for admissible magnetic potential vectors $\ba \in {\bf A}(\ba_0,M)$, stated in Theorem \ref{thm-stab}.
The proof essentially relies on the following:
\begin{lemma}
\label{lm-1}
Under the assumptions of Proposition \ref{pr-Carly}, let $y$ denote the ${\rm C}^0([-T,T];\HH_1)$-solution to \eqref{ia10} and put $\mathcal{I} :=\| {\rm e}^{-s \eta(0,.)} y(0,.) \|_{{\rm L}^2(\Omega)}^2$. Then there are three constants $\lambda_0>0$, $s_0>0$, and $C_3=C(T,M,\Omega, \Gamma^+, \lambda_0,s_0,u_0)>0$, such that we have
\beas
\mathcal{I} & \leq & C_3  s^{-1/2} \lambda^{-1} \left(  \int_{-T}^T \int_{\Gamma^+} {\rm e}^{-2s \eta(t,\sigma)} \varphi(t,\sigma) \partial_{\nu} \beta(t,\sigma) |\partial_{\nu} y(t,\sigma)|^2 {\rm d} \sigma  {\rm d} t \right. \nonumber \\
 & & \ \ \ \ \ \ \ \ \ \ \ \ \ \ \ \ \ \ \left. + s^{-1} \lambda^{-1} \| {\rm e}^{-s \eta(0,.)} ( \tilde{\ba}-\ba ) \|_{{\rm L}^2(\Omega)^n}^2 \right),
\eeas
for all $\lambda \geq \lambda_0$ and $s \geq s_0$.
\end{lemma}
\begin{proof}
Set $\psi := {\rm e}^{-s \eta} y$. Bearing in mind that $\psi(-T,.)=0$, we find out that
$$ \mathcal{I} = \int_{-T}^{0} \int_{\Omega} \partial_t |\psi(t,x)|^2 {\rm d} x {\rm d} t = 2\Pre{\int_{-T}^{0} \int_{\Omega} \psi'(t,x) \overline{\psi(t,x)} {\rm d} x {\rm d} t}, $$
whence
\beas
\mathcal{I} & = & 2\Pim{\int_{-T}^{0} \int_{\Omega} \left( {\rm i} \psi'(t,x)  +\Delta \psi(t,x) + s^2 |\nabla \eta(t,x)|^2 \psi(t,x)  \right) \overline{\psi(t,x)} {\rm d} x {\rm d} t}  \\
& =& 2 \Pim{\int_{-T}^0 \int_{\Omega} M_1 \psi(t,x) \overline{\psi(t,x)} {\rm d} t {\rm d} x}.
\eeas
Therefore
$|\mathcal{I}| \leq 2 \| M_1 \psi \|_{{\rm L}^2(Q_T)} \| \psi \|_{{\rm L}^2(Q_T)}$ from the Cauchy-Schwarz inequality. Since $\inf \{ \varphi(t,x),\ (t,x) \in Q_T \} >0$ by \eqref{defphieta}, there is thus a constant $C_4>0$ depending only on $T$ and $\lambda_0$, such that we have
$$ |\mathcal{I}| \leq C_4 s^{-3 \slash 2} \lambda^{-2} \left( s^3 \lambda^4 \| {\rm e}^{-s \eta} \varphi^{3 \slash 2} y \|_{{\rm L}^2(Q_T)}^2  + \| M_1 ( {\rm e}^{-s \eta} y) \|_{{\rm L}^2(Q_T)}^2 \right), $$
for all $s \geq s_0$ and $\lambda \geq \lambda_0$.
This, combined with \eqref{defiq} and Proposition \ref{pr-Carly}, yields
\beas
|\mathcal{I}| &\leq & C_5 s^{-1 \slash 2} \lambda^{-1} \left(  \sum_{\rho=w,y}\int_{-T}^{T} \int_{\Gamma^+} {\rm e}^{-2s \eta(t,\sigma)} \varphi(t,\sigma) \partial_{\nu} \beta(t,\sigma)  | \partial_{\nu} \rho(t,\sigma) |^2 {\rm d} \sigma {\rm d} t \right. \nonumber \\
& & \ \ \ \ \ \ \ \ \ \ \ \ \ \ \ \ \left. + s^{-1} \lambda^{-1} \| {\rm e}^{-s \eta} (\tilde{\ba}-\ba) \|_{{\rm L}^2(Q_T)^n}^2  \right),
\eeas
the constant $C_5>0$ depending on $\Omega$, $\Gamma^+$, $T$, $M$, $\lambda_0$, $s_0$ and $u_0$. Now the result follows immediately from this since $\eta(t,x) \geq \eta(0,x)$ for all $(t,x) \in Q_T$ by \eqref{defphieta}.
\end{proof}

Armed with Lemma \ref{lm-1}, we turn now to proving Theorem \ref{thm-stab}.
Since $\mathcal{I}= 4 \chi'(0)^2 \| {\rm e}^{{-s \eta}(0,.)} (\tilde{\ba}-\ba) . \nabla u_0 \|_0^2$ from \eqref{ia9}, Lemma \ref{lm-1} assures us that there are three constants  $\lambda_0>0$, $s_0>0$ and $C_6>0$, such that we have
\bea
& & C_6 \left( \| {\rm e}^{{-s \eta}(0,.)} (\tilde{\ba}-\ba) . \nabla u_0 \|_{{\rm L}^2(\Omega)}^2 - s^{-3/2} \lambda^{-2} \| {\rm e}^{{-s \eta}(0,.)} ( \tilde{\ba}-\ba ) \|_{{\rm L}^2(\Omega)^n}^2  \right) \nonumber \\
&\leq & s^{-1/2} \lambda^{-1} \sum_{\rho=w,y} \int_{-T}^T \int_{\Gamma^+} {\rm e}^{-2s \eta(t,\sigma)} \varphi(t,\sigma) \partial_{\nu} \beta(t,\sigma) |\partial_{\nu} \rho(t,\sigma)|^2 {\rm d} \sigma {\rm d} t,
\label{mc82}
\eea
for all $s \geq s_0$ and $\lambda \geq \lambda_0$.
In light of the definitions of $w$, $y$, and those of $u_{0,j}$, $u_j$, $\tilde{u}_j$, for $j=1,\ldots,n$, \eqref{mc82} entails
\beas
& & C_6 \left( \| {\rm e}^{{-s \eta}(0,.)} (\tilde{\ba}-\ba ) . \nabla u_{0,j} \|_{{\rm L}^2(\Omega)}^2 - s^{-3/2} \lambda^{-2} \| {\rm e}^{{-s \eta}(0,.)} ( \tilde{\ba}-\ba ) \|_{{\rm L}^2(\Omega)^n}^2  \right) \\
&\leq & s^{-1/2} \lambda^{-1} \int_{-T}^T \int_{\Gamma^+} {\rm e}^{-2s \eta(t,\sigma)} \varphi(t,\sigma) \partial_{\nu} \beta(t,\sigma) \left( \sum_{k=1,2}|\partial_{\nu} \partial_t^k (\tilde{u}_j-u_j)(t,\sigma)|^2 \right) {\rm d} \sigma {\rm d} t.
\eeas
Summing up the above estimate over $j=1,\ldots,n$, we get that
\bea
& &
C_6 \left( \| {\rm e}^{{-s \eta}(0,.)} DU_0 ( \tilde{\ba}-\ba ) \|_{{\rm L}^2(\Omega)^n}^2 - n s^{-3/2} \lambda^{-2} \| {\rm e}^{{-s \eta}(0,.)} (\tilde{\ba}-\ba) \|_{{\rm L}^2(\Omega)^n}^2  \right)
\nonumber \\
&\leq &
  s^{-1/2} \lambda^{-1} \int_{-T}^T \int_{\Gamma^+} {\rm e}^{-2s \eta(t,\sigma)} \varphi(t,\sigma) \partial_{\nu} \beta(t,\sigma) \left( \sum_{\tiny \begin{array}{c} j=1,\ldots,n \\ k=1,2 \end{array}}|\partial_{\nu} \partial_t^k (\tilde{u}_j-u_j)(t,\sigma)|^2 \right) {\rm d} \sigma {\rm d} t,  \nonumber \\
&&\label{mc100}
\eea
where the $n \times n$ real matrix $DU_0(x)$, for $x \in \Omega$, is the same as in \eqref{icc}. Notice that
we have
\bel{mc101}
\| DU_0(x) v \|_{\C^n} \geq \mu_1(x) \| v \|_{\C^n},\ x \in \Omega, \ v \in  \C^n,
\ee
where $\{ \mu_j(x) \}_{j=1}^n \subset  R_+^n$ denotes the non decreasing sequence of the singular values of $DU_0(x)$, and $\| v \|_{\C^n}$ stands for the Euclidian norm of $v$.
Moreover each $u_{0,j}$, $j=1,\ldots,n$, being taken in ${\rm H}^{p}(\Omega)$ with $p>n\slash 2+1$, it holds true that $u_{0,j} \in {\rm C}^1(\overline{\Omega})$, whence $\mu_1 \in {\rm C}^0(\overline{\Omega};\R_+)$ from \cite{K}[Theorem II.5.1]. This, combined with \eqref{icc}, yields $\mu:= \inf_{x \in \Omega} \mu_1(x) >0$, the constant $\mu$ depending on $\Omega$ and $\{ u_{0,j} \}_{j=1}^n$ only. As a consequence we have
$\| {\rm e}^{{-s \eta}(0,.)} DU_0 ( \tilde{\ba}-\ba ) \|_{{\rm L}^2(\Omega)^n} \geq  \mu \| {\rm e}^{{-s \eta}(0,.)} (\tilde{\ba}-\ba) \|_{{\rm L}^2(\Omega)^n}$ by \eqref{mc101}, and Theorem \ref{thm-stab} follows directly from this and \eqref{mc100} by choosing $s$ so large that $n s^{-3/2} \lambda_0^{-2} < \mu^2$.\\

\section*{Acknowledgement(s)}The authors are thankful to M. Bellassoued for fruitful discussions concerning this problem and valuable comments.

\end{document}